\documentclass[11pt,a4paper]{amsart}
\usepackage[margin=25mm]{geometry}
\usepackage[numbers]{natbib}
\usepackage{hyperref}

\usepackage{amssymb,amsmath}
\usepackage{amsthm}
\usepackage{bbm, bm}
\usepackage{stmaryrd}
\usepackage[usenames,dvipsnames]{xcolor}
\usepackage{color}
\usepackage{graphicx}
\usepackage{caption}
\usepackage{float}
\usepackage{wrapfig}

\usepackage{lineno}

\usepackage{abstract}

\input xy
\xyoption{all} 

\numberwithin{equation}{section}

\newtheorem{theorem}{Theorem}[section]
\newtheorem{corollary}[theorem]{Corollary}

\newtheorem{proposition}[theorem]{Proposition}

\newtheorem{observation}[theorem]{Observation}
\theoremstyle{definition}
\newtheorem{definition}[theorem]{Definition}
\newtheorem{example}[theorem]{Example}

\newenvironment{warning}[1][Warning.]{\begin{trivlist}
\item[\hskip \labelsep {\bfseries #1}]}{\end{trivlist}}

\newenvironment{remark}[1][Remark.]{\begin{trivlist}
\item[\hskip \labelsep {\bfseries #1}]  }{ \end{trivlist}}

\newcommand{\rmd}{\textnormal{d}}
\newcommand{\rme}{\textnormal{e}}
\newcommand{\rmh}{\textnormal{h}}

\DeclareMathOperator{\Vect}{Vect}

\DeclareMathOperator{\w}{w}

\font\black=cmbx10 \font\sblack=cmbx7 \font\ssblack=cmbx5 \font\blackital=cmmib10  \skewchar\blackital='177
\font\sblackital=cmmib7 \skewchar\sblackital='177 \font\ssblackital=cmmib5 \skewchar\ssblackital='177
\font\sanss=cmss10 \font\ssanss=cmss8 
\font\sssanss=cmss8 scaled 600 \font\blackboard=msbm10 \font\sblackboard=msbm7 \font\ssblackboard=msbm5
\font\caligr=eusm10 \font\scaligr=eusm7 \font\sscaligr=eusm5  \font\fraktur=eufm10
\font\sfraktur=eufm7 \font\ssfraktur=eufm5 
\font\bsymb=cmsy10 scaled\magstep2
\def\all#1{\setbox0=\hbox{\lower1.5pt\hbox{\bsymb
       \char"38}}\setbox1=\hbox{$_{#1}$} \box0\lower2pt\box1\;}
\def\exi#1{\setbox0=\hbox{\lower1.5pt\hbox{\bsymb \char"39}}
       \setbox1=\hbox{$_{#1}$} \box0\lower2pt\box1\;}

\def\tx#1{{\fam0\relax#1}}

\newfam\bifam
\textfont\bifam=\blackital \scriptfont\bifam=\sblackital \scriptscriptfont\bifam=\ssblackital

\newfam\blfam
\textfont\blfam=\black \scriptfont\blfam=\sblack \scriptscriptfont\blfam=\ssblack

\newfam\bbfam
\textfont\bbfam=\blackboard \scriptfont\bbfam=\sblackboard \scriptscriptfont\bbfam=\ssblackboard

\newfam\ssfam
\textfont\ssfam=\sanss \scriptfont\ssfam=\ssanss \scriptscriptfont\ssfam=\sssanss
\def\sss#1{{\fam\ssfam\relax#1}}

\newfam\clfam
\textfont\clfam=\caligr \scriptfont\clfam=\scaligr \scriptscriptfont\clfam=\sscaligr

\newfam\frfam
\textfont\frfam=\fraktur \scriptfont\frfam=\sfraktur \scriptscriptfont\frfam=\ssfraktur

\def\hpb#1{\setbox0=\hbox{${#1}$}
    \copy0 \kern-\wd0 \kern.2pt \box0}
\def\vpb#1{\setbox0=\hbox{${#1}$}
    \copy0 \kern-\wd0 \raise.08pt \box0}

\def\pmb#1{\setbox0\hbox{${#1}$} \copy0 \kern-\wd0 \kern.2pt \box0}
\def\pmbb#1{\setbox0\hbox{${#1}$} \copy0 \kern-\wd0
      \kern.2pt \copy0 \kern-\wd0 \kern.2pt \box0}
\def\pmbbb#1{\setbox0\hbox{${#1}$} \copy0 \kern-\wd0
      \kern.2pt \copy0 \kern-\wd0 \kern.2pt
    \copy0 \kern-\wd0 \kern.2pt \box0}
\def\pmxb#1{\setbox0\hbox{${#1}$} \copy0 \kern-\wd0
      \kern.2pt \copy0 \kern-\wd0 \kern.2pt
      \copy0 \kern-\wd0 \kern.2pt \copy0 \kern-\wd0 \kern.2pt \box0}
\def\pmxbb#1{\setbox0\hbox{${#1}$} \copy0 \kern-\wd0 \kern.2pt
      \copy0 \kern-\wd0 \kern.2pt
      \copy0 \kern-\wd0 \kern.2pt \copy0 \kern-\wd0 \kern.2pt
      \copy0 \kern-\wd0 \kern.2pt \box0}

\mathchardef\za="710B  
\mathchardef\zb="710C  
\mathchardef\zg="710D  
\mathchardef\zd="710E  
\mathchardef\zve="710F 
\mathchardef\zz="7110  
\mathchardef\zh="7111  
\mathchardef\zvy="7112 
\mathchardef\zi="7113  
\mathchardef\zk="7114  
\mathchardef\zl="7115  
\mathchardef\zm="7116  
\mathchardef\zn="7117  
\mathchardef\zx="7118  
\mathchardef\zp="7119  
\mathchardef\zr="711A  
\mathchardef\zs="711B  
\mathchardef\zt="711C  
\mathchardef\zu="711D  
\mathchardef\zvf="711E 
\mathchardef\zq="711F  
\mathchardef\zc="7120  
\mathchardef\zw="7121  
\mathchardef\ze="7122  
\mathchardef\zy="7123  
\mathchardef\zf="7124  
\mathchardef\zvr="7125 
\mathchardef\zvs="7126 
\mathchardef\zf="7127  
\mathchardef\zG="7000  
\mathchardef\zD="7001  
\mathchardef\zY="7002  
\mathchardef\zL="7003  
\mathchardef\zX="7004  
\mathchardef\zP="7005  
\mathchardef\zS="7006  
\mathchardef\zU="7007  
\mathchardef\zF="7008  
\mathchardef\zW="700A  
\mathchardef\zC="7009  

\newcommand{\be}{\begin{equation}}
\newcommand{\ee}{\end{equation}}

\newcommand{\bea}{\begin{eqnarray}}
\newcommand{\eea}{\end{eqnarray}}
\def\*{{\textstyle *}}
\newcommand{\R}{{\mathbb R}}

\newcommand{\s}{{\textstyle *}}

\def\Sec{\sss{Sec}}
\def\Vect{\sss{Vect}}

\def\sT{{\sss T}}

\def\xi{\tx{i}}

\def\cN{\cal N}

\def\s*{{\scriptstyle *}}

\def\cN{\mathcal{N}}
\def\cC{\mathcal{C}}

\newcommand{\beas}{\begin{eqnarray*}}
\newcommand{\eeas}{\end{eqnarray*}}

\title{On the bundle of null cones} 
\author{Andrew James Bruce } 
   \address{Department of Mathematics,
The Computational Foundry,
Swansea University Bay Campus,
Fabian Way,
Swansea, SA1 8EN}  
   \email{andrewjamesbruce@googlemail.com}

   \date{\today}
 
\begin{document}
 \maketitle
\vspace{-20pt}
\begin{abstract}{\noindent We examine the bundle structure of the field of nowhere vanishing null vector fields on a (time-oriented) Lorentzian manifold. Sections of what we refer to as the null tangent, are by definition nowhere vanishing null vector fields. It is shown that the set of nowhere vanishing null vector fields comes equipped with a para-associative ternary partial product. Moreover, the null tangent bundle is an example of a non-polynomial graded bundle.}\\

\noindent {\Small \textbf{Keywords:} Lorentzian manifolds;~null cones;~fibre bundles;~semi-heaps}\\
\noindent {\small \textbf{MSC 2020:} 20N10;~53B30;~53C50;~58A32;~83C99}
\end{abstract}

\section{Introduction} 
According to the theory of general relativity and closely related theories of gravity, spacetime is a four-dimensional Lorentzian manifold. An important aspect of Lorentzian geometry is the field of null cones due to their r\^{o}le in the causal nature of spacetime (see \cite{Hawking:1973,Kronheimer:1967}). In this work, we construct a fibre bundle of null cones (with the zero section removed) over a connected time-oriented four-dimensional Lorentzian manifold - we refer to this bundle as the \emph{null tangent bundle} (see Definition \ref{def:NullTanBun}).
\begin{wrapfigure}{r}{0.5\textwidth}
\vspace{-35pt}
\begin{center}
\includegraphics[scale=0.5, angle= 23]{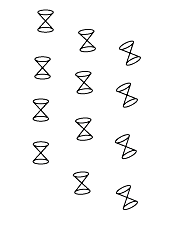}
\vspace{-35pt}
\caption{\small The null tangent bundle is constructed via attaching a split double cone to each point of a spacetime.}
\end{center}
\end{wrapfigure}
There is no Riemannian analogue of the null tangent bundle and such structures truly belong to Lorentzian geometry. Note that unit tangent bundles exist on any Lorentzian or Riemannian manifold.  We show that the null tangent bundle is an example of a natural bundle in the sense that it is canonically defined from the spacetime and requires no additional structure (see the proof of Proposition \ref{prop:SmoBun}). Moreover, the null tangent bundle of a spacetime comes with an action of the multiplicative group of strictly positive real numbers:  we refer to this action as the homothety (see Observation \ref{obs:RAction}). We have a kind of graded manifold. However, the admissible coordinate transformations are not polynomial, so we do not have a graded bundle (see \cite{Grabowski:2012}). Sections of the null tangent bundle are, by definition, nowhere vanishing future or past directed null vector fields. Recall that we do not have a vector space on the set of null vectors at any given point on spacetime. For instance, the sum of two future-directed null vectors is a timelike vector, unless the two vectors are parallel. However, we will show that the sections of the null tangent bundle come with a partial ternary operation that is para-associative, i.e., we have a partial semiheap structure (see Theorem \ref{trm:SemiHeap}). \par
 We remark that nowhere vanishing null vector fields, with other additional properties, are important in general relativity and supergravity.  For example, Robinson manifolds come equipped with a nowhere vanishing vector field whose integral curved are null geodesics (see \cite{Trautman:2002}). Other examples include pp-wave spacetimes, which are  Lorentzian manifolds that admit a covariantly constant null vector field. All these examples fall under the umbrella of null G-structures (see \cite{Papadopoulos:2019}).

\medskip

\noindent \textbf{Conventions:} We will consider four-dimensional Lorentzian manifolds $(M, g)$ of signature $(- , + , + , +)$. Recall that a vector $v \in \sT_p M$ is 
\begin{align*}
&\textnormal{timelike if} & g_p(v,v)< 0,\\
&\textnormal{null if} & g_p(v,v) =0,\\
&\textnormal{spacelike if} &  g_p(v,v)> 0.
\end{align*}
The set of null vectors at $p \in M$ forms the \emph{double cone} $\cN_p \subset \sT_pM$ (as a set). A vector field $\bm{v} \in \Vect(M)$ is said to be timelike/null/spacelike if at every point $p \in M$ the vector $\bm{v}|_p$ is timelike/null/spacelike.
\begin{definition}
A Lorentzian manifold $(M,g)$ is \emph{time-orientable} if it admits a timelike vector field $\bm{\bm{\tau}} \in \Vect(M)$.
\end{definition}
We say that a vector $v \in \sT_p M$ is future directed if $g_p(v, \bm{\tau} |_p) <0$ and past directed if $g_p(v, \bm{\tau} |_p) > 0$. Note that if $v$ is future/past directed then $-v$ is past/future directed.
\begin{definition}
A \emph{spacetime} is a connected time-oriented four-dimensional Lorentzian manifold $(M,g, \bm{\tau})$.
\end{definition}
Recall that by connected, we mean that $M$ is not homeomorphic to the union of two or more disjoint non-empty open subsets. A smooth manifold admits a Lorentzian metric if it admits a nowhere vanishing vector field. The existence of a nowhere vanishing vector field implies that $M$ must be non-compact or compact with zero Euler characteristic. 
\begin{example}
A \emph{globally hyperbolic spacetime} is of the form $M \simeq \R \times \Sigma$, where  $\Sigma$ is a three-dimensional smooth manifold.  The vector field $\partial_t$, where $t$ is the global coordinate on $\R$,  defines a timelike vector field and so $M$ is time-oriented. 
\end{example}
\begin{wrapfigure}{r}{0.5\textwidth}
\vspace{-10pt}
\begin{center}
\includegraphics[scale=0.2]{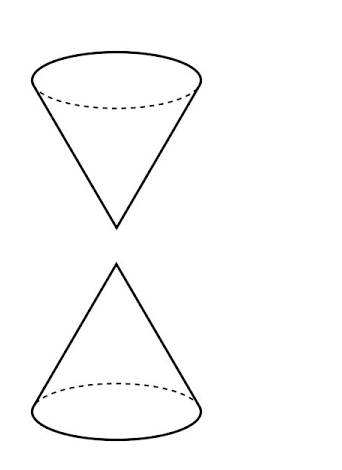}
\caption{\small The split null cone $\cC_0$.}
\end{center}
\end{wrapfigure}
Once a time-orientation has been chosen we can consistently decompose null cones at all points into the future and past directed components, $\cN_p = \cN^+_p \cup \cN^-_p$. At each point, we can remove the zero vector $\mathbf{0}_p \in \sT_p M$ to get the \emph{split null cone} $\cC_0 = \cC_0^+ \cup \cC_0^- \simeq \cN_p \setminus \{ \mathbf{0}_p\} \simeq \R^3_* \sqcup \R^3_*$. Note that the split null cone is a three dimensional smooth manifold with an atlas consisting of two charts $\big \{ (\cN^+_p , \phi_+), ~ (\cN^-_p, \phi_-)  \big\}$, with 
\begin{align}\label{eqn:ConAatlas}
\nonumber & \phi^{-1}_{\pm} \, :  ~\R^3_* \longrightarrow \cN^{\pm}_p\\
& (x,y,z)\mapsto ( \pm \sqrt{x^2 + y^2 + z^2}, x,y,z)\,.
\end{align}
We will employ vierbein fields in describing the tangent spaces. Let $\{(U_\alpha, \psi_\alpha) \}_{\in \mathcal A}$ be an atlas on $M$ and let $\mathcal{T}M$ be the tangent sheaf. Then for any $U \subset M$ in the chosen atlas we can employ vierbien fields, i.e., we have a basis $\{e_a = e_a^{\,\, \mu }(x)\partial_\mu \}$ of $\mathcal{T}M(U)$. Recall that the vierbeins satisfy $g^{\mu \nu} = e_a^{\,\, \mu} e_b^{\,\, \nu}\, \eta^{ba}$ where $\eta$ is the Minkowski metric. The basis fields transform under the (restricted) Lorentz group, i.e., $e_{a'} =  \Lambda_{a'}^{\,\, \, b} \, e_b$. The inverse vierbeins are defined via $\eta_{ab}=  e_a^{\,\, \mu} e_b^{\,\, \nu} \, g_{\nu \mu}$. Thus, using the natural basis $\{ e_a^{\,\, \mu}\partial_\mu|_p  \}$ of the tangent space $\sT_p M$ we have that $g_p(v,w) =  v^a w^b \eta_{ba}$. In particular, the components of any vector $v \in \sT_p M$ transform under Lorentz transformations as $v^{a'} = v^b \, \Lambda_b^{\,\, a'}$, where $\Lambda_b^{\,\, a'} \, \Lambda_{a'}^{\, \, a} = \delta_b^{\,\, a}$. 
\medskip 

\begin{remark}
One could further insist that a spacetime is parallelizable, i.e., we have a global basis for vector fields. It is physically reasonable that spacetimes be parallelizable.  However, we will not impose this condition in our definition of a spacetime.
\end{remark}
The algebraic structure of \emph{heaps} were introduced by Pr\"{u}fer \cite{Prufer:1924} and  Baer \cite{Baer:1929} as a set equipped with a ternary operation satisfying some natural axioms. A heap can be thought of as a group in which the identity element is forgotten. Given a group,  we can  construct a heap by defining the ternary operation as $(a,b,c) \mapsto ab^{-1}c$.   Conversely, by selecting an element in a heap, one can reduce the ternary operation to a  group operation, such that the chosen element is the identity element. \par 
There is a weaker notion of a \emph{semiheap} as  a non-empty set $H$, equipped with a ternary operation $[a,b,c] \in H$ that satisfies the \emph{para-associative law} 
\begin{equation}\label{eqn:ParaAss}
\big[ [a,b,c] , d,e \big] = \big [ a,[d,c,b],e\big] =  \big[ a,b,[c , d,e] \big]\, ,
\end{equation}
for all $a,b,c,d$ and $e \in H$. A semiheap is a \emph{heap} when all its elements are \emph{biunitary}, i.e., $[a,b,b] =a$ and  $[b,b,a] = a$, for all $a$ and $b \in H$. For more details about heaps and related structures the reader my consult Hollings \& Lawson \cite{Hollings:2017}

\section{The null tangent bundle}
\subsection{Construction of the null tangent bundle}\label{subsec:NullTan}
We now proceed to the main definition of this paper. We build the null tangent bundle following the classical construction of the tangent bundle or unit tangent bundle but now replacing the tangent spaces with spilt null cones.
\begin{definition}\label{def:NullTanBun}
Let $(M, g, \bm{\tau})$ be a spacetime, then the \emph{null tangent bundle} is defined as the disjoint union of split null cones, i.e.,
\begin{align*}
\cN M & =  \bigsqcup_{p \in M} \cN_p \setminus \{\mathbf{0}_p\} \\
&= \big\{ (p,v) ~~| ~~ p \in M, ~~ v \in \cN_p \setminus \{ \mathbf{0}_p\} \big\}\,,
\end{align*}
together with the natural projection 
\begin{align*}
\pi : & ~\cN M \longrightarrow M\\
& (p,v) \mapsto p\,.
\end{align*}
\end{definition}
\begin{warning}
The null tangent bundle is \emph{not} a vector bundle: this is clear as we do not have a zero section, and the sum of two (non-zero) null vectors is a timelike vector unless the pair of null vectors are linearly dependent. 
\end{warning}
We will on occasion use the fact that  $\cN M =  \cN M^+ \cup \cN M^- $, where $ \cN M^\pm :=  \bigsqcup_{p \in M} \, \cN_p^\pm \setminus \{ \mathbf{0}_p\}$. Clearly, $ \cN M^+ \cap \cN M^- = \emptyset$ and so $\cN M$ is  disconnected.
\begin{remark}
The null tangent bundle should be compared with the \emph{unit tangent bundle} of a (pseudo-)Riemannian manifold $\mathcal{U}M := \bigsqcup_{p \in M} \{ v \in \sT_p M ~ |~~ g_p(v,v) = 1 \}$. The fibres are diffeomorphic to $S^{n-1}$, assuming $M$ is of dimension $n$.
\end{remark}
\begin{proposition}\label{prop:SmoBun}
Let $(M, g, \bm{\tau})$ be a spacetime. The null tangent bundle $\cN M$ is a smooth fibre bundle $\pi : \cN M \rightarrow M$ with typical fibre $F_p \simeq \cC_0 = \cC_0^+ \cup \cC_0^- $.
\end{proposition}
\begin{proof}
The only part of the proposition that is not immediate is that the total space $\cN M$ is a smooth manifold. To show that the null tangent bundle of a spacetime is a (smooth) fibre bundle, we construct a natural bundle atlas of $\cN M $ inherited from an atlas of $M$. We will consider $\cN M$ as a natural bundle (see \cite{Kolar:1993}).  Let $\{ (U_ \alpha, \, , \psi_\alpha)\}_{\alpha \in \mathcal{A}}$ be an atlas of $M$ (not necessarily a maximal atlas).  From the definition of the null tangent bundle $(p, v_\pm =v_\pm^a e_a)\in \pi^{-1}(U_\alpha)$, using vierbeins. Here ``$\pm$'' signifies if the vector is future or past directed.  We then define
\begin{align*}
& \pi^{-1}(U_\alpha) \stackrel{\phi_\alpha}{\longrightarrow} \psi_\alpha(U_\alpha) \times \R^3_* \sqcup \R^3_* \\
& \phi_\alpha(p, v_\pm) = (\psi_\alpha(p), v_\pm^i) = (x^\mu, v_\pm^i)\,,
\end{align*}
where $v_\pm^i$, now interpreted as coordinates on $\cN_p^\pm \setminus \{ \mathbf{0}_p\}$, are the spacial components of $v_\pm$.   The inverse map is  (see equation \eqref{eqn:ConAatlas})
$$\phi^{-1}_\alpha (x^\mu , v_\pm^i) =  (p, ~\pm \sqrt{v_\pm^j v_\pm^k\delta_{kj}}, ~v_\pm^i)\,.$$
Thus, our natural choice of atlas is $\{(U_\alpha \times \cC_0^\pm, \phi^\pm_\alpha) \}_{\alpha \in \mathcal{A}}$ where
\begin{align*}
\phi^\pm_\alpha &:~  U_\alpha \times \cC_0^\pm \longrightarrow \psi_\alpha(U_\alpha) \times \R^3_*\\
& (p, v_\pm)  \mapsto(x^\mu, v_\pm^i)\,.
\end{align*}
The admissible changes of coordinates are 
\begin{equation}\label{eqn:CordChanges}
x^{\mu'} =  x^{\mu'}(x), \qquad v_\pm^{i'} = v_\pm^i \, \Lambda_i^{\, \, i'}(x) + \sqrt{v^j_\pm v^k_\pm\delta_{kj}}\, \Lambda_0^{\, \, i'}(x)\,.
\end{equation}
Via inspection, we observe that these coordinate transformations are differentiable and give a bundle atlas.
\end{proof}
Recall that a Weyl transformation is a local rescaling of the metric, i.e., 
$$g \mapsto g' := \rme^{-2 \, f} \, g\,$$
where $f \in C^\infty(M)$ is an arbitrary smooth function. Weyl transformations preserve split null cones at any given point on $M$ as the Weyl factor is strictly positive. Time-orientation is preserved under Weyl transformations.  We are then led to the following.
\begin{observation}
Let $(M, g , \bm{\tau})$ be a spacetime. The null tangent bundle $\cN M$ only depends on the conformal class of $g$.
\end{observation}
The null tangent bundle can, as standard for any fibre, be restricted to an immersed or embedded submanifold $N \subset M$. The definition is clear:
\begin{align*}
\cN M|_N & =  \bigsqcup_{p \in N} \cN_p \setminus \{\mathbf{0}_p\} \\
&= \big\{ (p,v) ~~| ~~ p \in N, ~~ v \in \cN_p \setminus \{ \mathbf{0}_p\} \big\}\,,
\end{align*}
together with the restriction of the natural projection 
\begin{align*}
\pi|_N : & ~\cN M|_N \longrightarrow N\\
& (p,v) \mapsto p\,.
\end{align*}
The local trivialisations of $\cN M$ restrict to local trivialisations
$$\phi_\alpha|_N : ~ (\pi|_N)^{-1}\big(U_\alpha \cap N \big) \longrightarrow (U_\alpha \cap N) \times \R^3_* \sqcup \R^3_*\,.$$
\begin{observation}
As $\cN M$ has disconnected fibres, a smooth (in particular, continuous) global or local section lies in one of the connected components, i.e., is a nowhere vanishing future or past directed null vector field. 
\end{observation}
\subsection{The homothety structure}
Both future and past directed null vectors (at any point) can be rescaled by a non-zero positive number $\lambda \in \R_{>0} := (0, \infty)$ and this does not change the null cone. We thus make the following observation.
\begin{observation}\label{obs:RAction}
Let $M$ be a spacetime. The null tangent bundle $\cN M$ comes with a smooth action of the multiplicative group of (strictly) positive real numbers $(\R_{>0}, \cdot)$
$$\rmh : ~ \R_{>0} \times \cN M \longrightarrow \cN M\,,$$
given in local coordinates by
$$\rmh^*_\lambda (x^\mu, v_\pm^i) = (x^\mu , ~ \lambda \, v_\pm^i)\,.$$
\end{observation}
We remind the reader that $(\R_{>0}, \cdot)$ is a Lie group. Note that as $\lambda$ is positive and non-zero the action is well-defined, i.e., is consistent with the  changes of coordinates (see \eqref{eqn:CordChanges}). Thus, we can view $\cN M$ as a non-polynomial positively graded manifold by assigning weights $\w(x^\mu) = 0$ and $\w(v_\pm^i) = 1$.  The action of $(\R_{>0}, \cdot)$ on $\cN M$ we refer to as the \emph{homothety}. Note that the homothety acts trivially on $M$ and preserves the components $\cN M^\pm$.
\begin{remark}
The null tangent bundle is \emph{not} a graded bundle in the sense of Grabowski and  Rotkiewicz (see \cite{Bruce:2017,Grabowski:2009,Grabowski:2012,Jozwikowski:2016}). In particular, the homothety cannot be extended to a smooth action of the multiplicative semigroup $(\R_{\geq 0}, \cdot)$.  We thus have a more general non-negatively graded manifold than usually considered in the literature, i.e., the coordinate transformations are non-polynomial (see \cite{Voronov:2002} and therein for further references).
\end{remark}
\subsection{Triviality of the null tangent bundle}
It is well known that fibre bundles over contractable spaces are topologically trivial: this directly implies the following.
\begin{observation}
Let $(\R^4, \eta)$ be Minkowski spacetime equipped with its canonical time-orientation. Then $\cN\R^4$ is topologically trivial, i.e., $\cN \R^4 \stackrel{\sim}{\rightarrow} \R^4 \times \cC_0$. 
\end{observation}
 This result extends to a large class of spacetimes.
\begin{proposition}
Let $(M, g, \bm{\tau})$ be a spacetime such that $M$ is parallelizable. Then $\cN M$ is topologically trivial, i.e., $\cN M \stackrel{\sim}{\rightarrow} M \times \cC_0$. 
\end{proposition}
\begin{proof}
As $M$ is parallelizable, there exists a global orthonormal frame $\{ e_a\}$ of the module of vector fields. We then define a bundle map
\begin{align*}
\Phi & ~: M \times \cC_0 \longrightarrow \cN M \\
& (p, v_\pm^i) \mapsto (p, \pm \sqrt{v_\pm^j v_\pm^k \delta_{kj}} \, e_0|_p + v^i_\pm \, e_i|_p)\,.  
\end{align*}
It is clear that such a map is a diffeomorphism, thus $\Phi^{-1}: \cN M \stackrel{\sim}{\rightarrow} M \times \cC_0$.
\end{proof} 
\begin{corollary}
If $M$ is parallizable, then the set $\Sec(\cN M)$ is non-empty.
\end{corollary}
\begin{example}[The Schwarzchild vacuum]
Consider the manifold $M =  \R \times (1,\infty) \times S^2$ equipped with Schwarzchild coordinates $(t,r , \theta, \varphi)$, where $-\infty  < t < \infty $, $1 < r < \infty$, $0 < \theta < \pi$, and $-\pi < \varphi < \pi$. In natural units $c=G = 1$ and $M = 1\slash 2$, the metric is diagonal and given by
\smallskip
\begin{equation*}
\begin{pmatrix}
 -\left(1-\frac{1}{r} \right) & 0 & 0 & 0\\
0 & \left(1-\frac{1}{r} \right)^{-1} & 0 & 0\\
0 & 0 & r^2 & 0\\
0 & 0 & 0 & r^2 \, \sin^2 \theta
\end{pmatrix}
\end{equation*}
The Schwarzchild radius is then $r_S = 1$. The coordinate vector $\bm{\tau} =  \frac{\partial}{\partial t}$ provides the time-orientation. The vierbein fields are: 
\begin{align*}
& e_t^{\,\, t} =  \frac{1}{\sqrt{1-\frac{1}{r}}}\, , & e_r^{\,\, r} =  \sqrt{1-\frac{1}{r}}\, ,\\
& e_\theta^{\,\, \theta} =  \frac{1}{r}\, , & e_\varphi^{\,\, \varphi}=  \frac{1}{r \, \sin \theta}\,,
\end{align*}
all other fields vanish.  The map $\Phi : M \times \cC_0 \rightarrow \cN M$ is given by
\begin{align*}
(p, v_\pm^r ,v_\pm^\theta, v_\pm^\varphi ) & \mapsto \left(p, \pm \sqrt{(v^r_\pm)^2 + (v^\theta_\pm)^2 + (v^\varphi_\pm)^2} \,\big(\sqrt{1 - 1/r}\big)^{-1} \, \partial_t|_ p \right.\\
& \left.+ v^r_\pm  \, \sqrt{1 - 1/r}\, \partial_r|_p 
 + v^\theta_\pm /r \,\partial_\theta |_p(v^r_\pm)^2 + v^\varphi_\pm /( r\, \sin \theta) \, \partial_\varphi|_p \right)\,.
\end{align*}
\end{example}
\subsection{The canonical distribution}
We will, for simplicity, just consider the future-directed null cone.  We denote the projection as $\pi : \cN M^+ \rightarrow M$. Then $v \in \cN M^+$ is a future-directed (non-zero) null vector and we set $p:= \pi(v) \in M$.  There is then a canonical linear map 
$$\rmd \pi_v :  \sT_v \cN M^+ \longrightarrow \sT_p M\, ,$$
given by 
$$\sT_v \cN M^+ \ni \textrm{x} = \textrm{x}^a e_a|_p + \textrm{x}^i\partial_{v^i}|_v \mapsto \rmd \pi_v(\textrm{x}) =  \textrm{x}^a e_a|_p \in \sT_p M \,.$$
We then construct a one-form via $\Theta_v : \sT_v \cN M^+ \rightarrow \R$, defined as $\Theta_v (\textrm{x}) := g_p\big(v , \rmd \pi_v(\textrm{x}) \big)$. Using vierbeins we can write
$$\Theta_v(\textrm{x}) = {-}\sqrt{v^l v^k \delta_{kl}}\,  \textrm{x}^0 + v^i \textrm{x}^j \delta_{ji}\,.$$
As the metric $g$ is non-degenerate and every $v$ is non-zero, $\Theta$ is nowhere vanishing, and this the kernel defines a distribution of rank $6$, i.e., a vector subbundle of $\sT \cN M^+$.  We are then led to the following definition.
\begin{definition}
Let $(M, g, \bm{\tau})$ be a spacetime. The associated \emph{canonical distribution} is defined as 
$$D :=  \ker ( \Theta ) = \big \{ (v, \mathrm{x}) \in  \sT \cN M^+ ~ |~~\Theta_v(\mathrm{x}) =0  \big\} \subset \sT \cN M^+\,.$$
\end{definition}
As a one-form  (locally) we can use the dual vierbiens (defined by $e^b e_a = e^b_\mu e^\mu_a  = \delta^a_{\, \, b}$) to write
\begin{align*}
\Theta & = - e^0 \, \sqrt{v^l v^k \delta_{kl}} + e^i v^j \delta_{ji}\\
&=  \rmd x^\mu \big(- e_\mu^0(x) \, \sqrt{v^l v^k \delta_{kl}} + e_\mu^i(x)v^j \delta_{ji} \big)\,.
\end{align*}
Changing perspective slightly and thinking of the distribution in terms of vector fields, we have a local basis of the canonical distribution  given by 
\begin{equation*}
 V_i = \frac{\partial}{\partial v^i}\,, \qquad 
X^j = \left( v^j e_0^\mu(x) + \delta^{jk}\sqrt{v^l v^m \delta_{ml}}\, e_k^\mu(x)   \right ) \frac{\partial}{\partial x^\mu}\,.
\end{equation*}
 \subsection{The null tangent prolongation}
 Let $I \subseteq \R$ be a connected nondegenerate\footnote{an interval is nondegenerate if it contains more than one point.}  interval. A curve is an smooth immersion $\gamma : I \rightarrow M$, i.e., the tangent map $\rmd_t \gamma : \sT_t I \rightarrow \sT_{\gamma(t)}M$ is one-to-one.  Recall that a curve $\gamma$  is said to be \emph{regular} if the curve has non-zero derivative at all points. That is, at any point the $\gamma(t)\in  M$ the associated velocity vector is non-zero. We will generally assume that all curves are regular. A curve is said to be a \emph{null curve} if the tangent vectors at all points on the curve are null vectors. A null curve is said to be \emph{future/past directed} if for all points on the curve the associated tangent vectors are future/past directed.
 \begin{definition}
 Let $\gamma : I \rightarrow M $ be a null regular curve. Then the \emph{null tangent prolongation} is the regular curve $ \sss{n t}\gamma : I \rightarrow \cN M$, given in local coordinates 
 $$\sss{n t}\gamma^*(x^\mu , v^i) = \left(x^\mu(t) , \frac{\rmd x^\nu(t)}{\rmd t}\, e_\nu^i\big (x(t) \big) \right)\,.$$
 \end{definition}
 \begin{definition}
 An \emph{implicit null differential equation} is a submanifold $D \subset \cN M$. A \emph{null solution} is a regular null curve whose null tangent prolongation takes values in $D$, i.e., $\sss{n t} \gamma(I) \subset D$. \par 
 If $D$ is the image of a null vector field $\bm{k}$, i.e., $\bm{k}(M) = D$, then we have an \emph{explicit null differential equation}.
 \end{definition}
 \subsection{The partial semiheap structure}
 Although sections - local or global - of the null tangent bundle do not form a vector space or similar, there is a partial ternary operation, in fact, a partial semiheap (see \cite{Hollings:2017} for details of heaps and related ternary structures). We restrict our attention to the global case, and so we consider parallelizable spacetimes in this subsection.
 \begin{theorem}\label{trm:SemiHeap}
 Let $(M, g, \bm{\tau})$ be a  parallelizable spacetime. Then the set of sections $\Sec(\cN M)$ of the null tangent bundle comes canonically equipped with the structure of a partial semiheap given by 
 $$[\bm{u},\bm{v}, \bm{w}] := \bm{u} \, g(\bm{v},\bm{w})\,,$$
 which is only defined when $\bm{v}_p$ and $\bm{w}_p$ are not proportional to each other for all points $p\in M$.
 \end{theorem}
\begin{proof}
The condition that the vectors $\bm{v}_p$ and $\bm{w}_p$ not be proportional to each other at all points $p \in M$ ensures, as we are in four dimensions, that $g(\bm{v},\bm{w})\in C^\infty(M)$  is nowhere vanishing -  via the intermediate value theorem we know that this function does not change sign.  We observe that $\Sec(\cN M)$ is closed under the ternary operation (when it is defined). \par 
The para-associativity (see \eqref{eqn:ParaAss}) follows directly by calculation, i.e.,
$$[[\bm{u}_1, \bm{u}_2, \bm{u}_3], \bm{u}_4, \bm{u}_5] =[\bm{u}_1, [\bm{u}_4, \bm{u}_3, \bm{u}_2], \bm{u}_5] = [\bm{u}_1, \bm{u}_2,[ \bm{u}_3, \bm{u}_4, \bm{u}_5]] =  \bm{u}_1 g(\bm{u}_2, \bm{u}_3)g(\bm{u}_4, \bm{u}_5)\,,$$
where $\bm{u}_2$ and $\bm{u}_3$, together with $\bm{u}_4$ and $\bm{u}_5$, are not proportional at all points. 
\end{proof}
\begin{corollary}
For any point $p \in M$, the split null cone $\cN_p \setminus \{\mathbf{0}_p\}$ has the structure of a partial semiheap.
\end{corollary}
\begin{corollary}
By fixing some section $\bm{w} \in \Sec(\cN M)$, we have an associative partial binary product 
$$\bm{u}_1 \bullet_{\bm{w}} \bm{u}_2 := [\bm{u}_1, \bm{w}, \bm{u}_2] = \bm{u}_1 g(\bm{w}, \bm{u}_2)\,,$$
that is, we have the structure of a partial semigroup.
\end{corollary}
\begin{example}
PP-wave spacetimes are spacetimes that admit a covariantly constant nowhere vanishing null vector field $\bm{k} \in \Sec(\cN M)$. The time orientation is chosen so that $\bm{k}$ is future-directed. Hence, PP-wave spacetimes naturally come with a partial semigroup structure on the set of nowhere vanishing null vector fields. 
\end{example}
\begin{example}
Robinson manifolds are Lorentzian manifolds that come equipped with a nowhere vanishing null vector field $\bm{k} \in \Sec(\cN M)$, whose integral curved are null geodesics. Hence, Robinson manifolds naturally come with a  partial semigroup structure on the set of nowhere vanishing null vector fields.
\end{example}
 \begin{remark}
 Let $(M, g, \bm{\tau})$ be a spacetime, then we have a semiheap structure on the set of vector fields, in fact, we have a para-associative ternary algebra (see \cite{Bruce:2022,Kerner:2008,Kerner:2018}). The ternary product is given by $[\bm{x}, \bm{y}, \bm{z}] := \bm{x} g(\bm{y}, \bm{z})$. As we have a given privileged vector field, there is an associative noncommutative algebra on the set of vector fields, given by   $\bm{x} \bullet_{\bm{\tau}} \bm{y} := [\bm{x}, \bm{\tau}, \bm{y}] = \bm{x} g(\bm{\tau}, \bm{y})$.
 \end{remark}
 Associated with any section $\bm{u}\in \Sec(\cN M)$ is a null congruence, i.e., a set of integral curves.  Note that if $f \in C^\infty(M)$ is a nowhere vanishing function, then $f \bm{u}$ generates the same null congruence. We can then define a partial semi-heap structure on the set of null congruences via their generating nowhere vanishing null vectors. 
\begin{observation}
Let $\bm{u}, \bm{v}$ and $\bm{w} \in \Sec(\cN M)$, and $f \in C^\infty(M)$ be nowhere vanishing. Then $f \bm{u}, f^{-1}\bm{v}$ and $ f \bm{w} \in\Sec(\cN M)$, and we have a module-like distribution rule
$$f[\bm{u}, \bm{v}, \bm{w}] = [f \bm{u}, f^{-1}\bm{v}, f \bm{w}]\,,$$
when $[\bm{u}, \bm{v}, \bm{w}]$ is defined.
\end{observation}
The multiplicative group of nowhere vanishing functions on a spacetime, we denote the underlying set as  $C_{inv}^\infty(M)$, has the structure of heap with the ternary product being given by 
$$\{f_1, f_2, f_3 \} := f_1 f_2^{-1} f_3\,.$$
Note that $\{f_1, f_2, f_3 \}^{-1} =f_1^{-1} f_2 f_3^{-1}\,. $
Using previous observation we are led to the following.
\begin{proposition}
Let $(M, g, \bm{\tau})$ be a spacetime. Then the partial semiheap structure on $\Sec(\cN M)$ and the heap structure on $C_{inv}^\infty(M)$ satisfy the following distribution rule (when defined)
$$\{f_1, f_2, f_3 \} [\bm{u}, \bm{v}, \bm{w}] = [\{f_1, f_2, f_3 \}\bm{u}, \{f_1, f_2, f_3 \}^{-1}\bm{v}, \{f_1, f_2, f_3 \}\bm{w}]\,, $$
for all $f_1, f_2, f_3 \in C_{inv}^\infty(M)$ and $\bm{u}, \bm{v}, \bm{w} \in \Sec(\cN M)$ where $\bm{v}_p$ and $\bm{w}_p$ are not proportional to each other for all points $p\in M$.
\end{proposition}

\section{Final remarks}
We have constructed the null tangent bundle and explored some of its immediate mathematical properties. Interestingly, we have a generalisation of a graded bundle in which the coordinate transformations are not polynomial. Moreover, the sections of the null tangent bundle, so nowhere vanishing null vector fields, come with the structure of a partial semiheap. These observations suggest that bundles of semiheaps and wider classes of graded manifolds should be studied. Particular focus should be on finding more geometric examples of ternary structures. 

\section*{Acknowledgements} 
The author thanks Janusz Grabowski, Steven Duplij and Damjan Pi\v{s}talo for their comments on earlier drafts of this paper.  

\end{document}